\documentclass[english]{article}
\usepackage[T1]{fontenc}
\usepackage[latin9]{inputenc}
\usepackage{geometry}
\usepackage{amsthm}
\usepackage{amsmath}
\usepackage{amssymb}
\usepackage[dvips]{graphicx}
\usepackage{pstricks}
\usepackage[square,comma,sort&compress,numbers]{natbib}
\graphicspath{{images/}}

\makeatletter
\theoremstyle{plain}
\newtheorem{thm}{\protect\theoremname}
  \theoremstyle{plain}
  \newtheorem{lem}[thm]{\protect\lemmaname}
 \theoremstyle{definition}
 \newtheorem*{defn*}{\protect\definitionname}
  \theoremstyle{plain}
  \newtheorem{cor}[thm]{\protect\corollaryname}

\makeatother

\usepackage{babel}
  \providecommand{\corollaryname}{Corollary}
  \providecommand{\definitionname}{Definition} 
  \providecommand{\lemmaname}{Lemma}
\providecommand{\theoremname}{Theorem}

\makeatother

\usepackage{babel}
  \providecommand{\definitionname}{Definition}
  
  \providecommand{\lemmaname}{Lemma}
  \providecommand{\theoremname}{Theorem}

\begin{document}
\bibliographystyle{plainnat}
\title{Analytical Solution of the Forward Displacement Problem for Spherical
Parallel Manipulators}

\author{Jose Rodriguez\\
Department of Mathematics, University of California at Berkeley\\
970 Evans Hall, Berkeley, California, 94720\\
jo.ro@berkeley.edu\\
 \\
 Maurizio Ruggiu \\
 Department of Mechanical Engineering, University of Cagliari \\
 Piazza d'Armi - 09123 Cagliari, Italy \\
 e-mail: ruggiu@dimeca.unica.it}
\maketitle

\begin{abstract}
In this paper, an analytical method that solves the \textit{Forward displacement problem (FDP)} of several common spherical parallel manipulators (SPMs) is presented. 
The method uses the quaternion algebra to express the problem as a system of equations and uses the Dixon determinant procedure to solve it. 
In addition, a case study is proposed for a specific SPM, which satisfies certain geometric conditions. Namely,  the SPM having  $3-\hat{\underline{R}}(RRR)_E\hat{R}$ architecture, with  $R$ denoting a revolute joint, $\hat{R}$'s denoting intersecting joints, $P$  denoting a prismatic joint, underlines indicating  the actuated joint,  and $(~)_E$ indicating a plane chain.

\end{abstract}
\section*{Keywords}  
Quaternions, Dixon determinant, forward displacement problem, spherical parallel manipulators.
\section{Introduction}
A \textit{spherical parallel mechanism (SPM)} refers  to a moving platform 
rotated about a fixed point by manipulators with three degrees of freedom.    
This fixed point is  designated as the center of spherical motion. 
This type of mechanism may find numerous applications in orientating devices and modelings wrists by exploiting some advantages in parallel architecture. 
Such advantages include  accuracy, repeatability in positioning, and excellent dynamic properties. 
In the literature, numerous SPM architectures have been proposed~\cite{Birglen2002,DiGregorio2004,DiGregorio2001,Gosselin1994,Li2002}. In addition,  studies covering a variety of relevant problems (workspace modeling, dexterity evaluation, design and optimization, singularity analysis and type synthesis) have been reported~\cite{Deidda2010,Gosselin1993,DiGregorio2001b,Gosselin1989,Liu2000,Leguay1997,Asada1985,Karouia2000,Kong2004}. 
The \textit{forward displacement problem (FDP)} of SPMs is concerned with finding the orientation of the moving platform for a given set of actuated-joint-variable values. 
However, due to their multi-loop architecture, there does not exist a closed form solution to the general problem. 
Moreover, the presence of transcendental equations suggests a high computational complexity. 
As noted in~\cite{Innocenti1993,Gosselin1994},  the problem admits at most eight solutions, from the roots of an eighth-order polynomial equation, known as the \textit{robot characteristic equation}.
Gosselin \textit{et al.}~\cite{Gosselin1994} proposed a method solving the FDP for SPMs with only revolute joints. 
In this method, the orientation of the end-effector is described by  Euler angles. 
Solutions in the form of an eighth-order polynomial equation were found. A similar solution was reported by Huang and Yao~\cite{Huang1999}, who regarded the direction cosines of each joint axis as functions of the actuated-joint variables. 
Innocenti and Parenti-Castelli~\cite{Innocenti1993} solved the FDP for SPMs with prismatic joints deriving a system of two equations, one eighth-order polynomial and one linear equation. 
Recently, Bai \textit{et al.}~\cite{Bai2009} revisited the problem, proposing a method based on the input-output equation of spherical four-bar linkages. 
\\
In this paper,  the closed-loop kinematic chain of a SPM is partitioned into two 
four-bar spherical chains leading to a trigonometric equation in the joint angles, which is solved semi-graphically  to obtain the joint variables for the determination of the moving platform orientation. 
\\
The idea behind this paper is to find an analytical method  to solve the FDP of SPMs  in radicals that obey a certain geometric condition.
The method proposed,  exploits the quaternion algebra and the Dixon determinant procedure. 
This approach benefits the FDP in the following respects:  
\\
\emph{(i)} The quaternion algebra is well suited for the spherical kinematics because it does not have the singularities or transcendental equations any  three parameter representation gives.
\\\emph{(ii)} The Dixon method is computationally inexpensive for a fixed set of parameters so a given system may have its solutions calculated efficiently.
\\\emph{(iii)} In the case of symbolic parameters, the Dixon method outputs a twenty by twenty matrices whose determinant are the coordinates of the solutions. 
\\\emph{(iv)} The FDP  of some SPM's may be solved in terms of radicals using these methods. 

\section{Problem Formulation}
The procedure proposed in the paper can be applied to SPMs subjected to the following geometric condition for each leg:
\begin{equation}
\mathbf{w}_i^T\mathbf{v}_i=\mathcal{C}_i~~~(i=1,2,3) \label{EQ_geometrical_condition}
\end{equation}
where $\mathbf{w}_i$ are vectors in $\mathbb{R}^3$  denoting the orientation of the axes of the motor joints. See Figure 1 for an example.
The 
 $\mathbf{v}_i$ are vectors in  $\mathbb{R}^3$  denoting  the orientation of the axes of the moving platform, and the  $\mathcal{C}_i$'s are constants. 
Equation (~\ref{EQ_geometrical_condition}) and others  are satisfied in the $3-\underline{R}RR$~\cite{Gosselin1994} and in the $3-\hat{R}(R\underline{P}R)\hat{R}$~\cite{Innocenti1993}.
In equation (~\ref{EQ_geometrical_condition}), the vectors are both expressed in a fixed orthonormal reference system with origin coinciding with the center of the spherical rotation. 
However, we define a reference system $\left\lbrace m\right\rbrace$ on top of the moving platform such that $\mathbf{v}_i=\mathbf{Q}\mathbf{\nu}_i$ ($\mathbf{\nu}_i$ is the vector $\mathbf{v}_i$ expressed in $\left\lbrace m\right\rbrace$). 
The matrix $\mathbf{Q}$ is  an orientation matrix with entries determined by the unit quaternions $\mathcal{Q}$~\cite{BottemaBook}
\begin{equation}
\mathcal{Q}=q_0+q_1\mathbf{e}_1+q_2\mathbf{e}_2+q_3\mathbf{e}_3, \nonumber
\end{equation}
with  $q_j$ $(j=0,...,3)\in\mathbb{R}$ and $1,\mathbf{e}_1,\mathbf{e}_2,\mathbf{e}_3$  a basis for the quaternion algebra.
It can be shown $\mathcal{Q}$  is expressed with entries in the real numbers $q_1,q_2,q_3,q_4$ as follows:
\begin{equation}
\mathbf{Q}=\left[\begin{array}{ccc} {(q_0^2+q_1^2-q_2^2-q_3^2)} & {2(q_1q_2-q_0q_3)} & {2(q_0q_2+q_1q_3)} \\ {2(q_1q_2+q_0q_3)} & {(q_0^2-q_1^2+q_2^2-q_3^2)} & {2(q_2q_3-q_0q_1)} \\ {2(q_1q_3-q_0q_2)}  & {2(q_0q_1+q_2q_3)} &  {(q_0^2-q_1^2-q_2^2+q_3^2)}  \end{array}\right] .\nonumber 
\end{equation}
In addition to the three geometric conditions, the use of unit quaternions gives a fourth normalizing condition. 
Indeed, every quaternion $\mathcal{Q}$ has a conjugate  $\mathcal{\tilde{Q}}=\left[\begin{array}{c c c c} {q_0} & {-q_1} & {-q_2} & {-q_3} \end{array}\right]$, which satisfies
\begin{equation}
\mathcal{Q}\mathcal{\tilde{Q}}=q_0^2+q_1^2+q_2^2+q_3^2=1 \label{EQ_unit_quaternion_definition}.
\end{equation}
While there are several known  representations  of $\mathbf{Q}$, this one using four parameters,  $q_0$, $q_1$, $q_2$, $q_3$,  offers robustness against the singularities  arising in any three parameter representation.
Equations (~\ref{EQ_geometrical_condition}), (~\ref{EQ_unit_quaternion_definition}) define a system of four quadric polynomial equations $f_i=0$,~$(i=1,...,4)$ with 4 unknowns: $q_j$,~$(j=0,...,3)$ providing the orientation of the moving platform when the motor angles are given (\emph{i.e.}, solution of the FD problem). 

We have rephrased our kinematics problem into the problem of solving a system of four quadrics in four unknowns $q_{0},q_{1},q_{2},q_{3}$. In this paper the \textit{Dixon determinant} method is used. 
We now show how to derive a \emph{Dixon matrix}, which vanishes at  coordinates of the solutions. \\

\section{Solution of the System of Polynomial Equations}


\subsection*{Notation}

In 1908, Dixon developed a criterion~\cite{Dixon1908} to determine when four
quadrics in three unknowns vanish at a common point.
 This criterion
claims the four quadrics vanish at a common point, if  a certain $20\times20$ matrix is singular. 
We call this matrix a Dixon matrix.  Its entries  will be determined by the coefficients of twenty specially chosen cubics in three unknowns. 
Given a  system of four quadrics in four unknowns, the same criterion can be used to determine the coordinates 
of solutions. By considering different combinations of the coordinates, we have a finite list which contains our solution set.  
Using this technique, we illustrate  the $3-\hat{\underline{R}}(RRR)_E\hat{R}$ 
case followed by an explicit numerical example.\\
\\
We set notation by recalling  algebraic geometry facts from~\cite{CoxBook}
to derive a Dixon matrix.
Let $f_{1},\dots,f_{s}$ be in the polynomial ring $\mathbb{C}\left[x_{1},\dots,x_{n}\right]$.
We set 
\[
{\bf V}\left(f_{1},\dots,f_{s}\right)=\left\{ \left(a_{1},\dots,a_{n}\right)\in\mathbb{C}^{n}\mid f_{i}\left(a_{1},\dots,a_{n}\right)=0\,\, for\,i=1,2,\dots ,m\right\}.
\]
and say  ${\bf V}\left(f_{1},\dots,f_{s}\right)$ is the affine variety
defined by $f_{1},\dots,f_{s}$ over $\mathbb{C}$. 
The affine variety ${\bf V}\left(f_{1},\dots,f_{s}\right)$ is the set of points in $\mathbb{C}^n$ which are solutions to the system   defined by $f_{1},\dots,f_{s}$. 
We set
$\langle f_{1},\dots,f_{s}\rangle=\left\{ \sum_{i=1}^{s}h_{i}f_{i}\mid h_{i}\in\mathbb{C}\left[x_{1},\dots,x_{n}\right]\right\} $
and call $\langle f_{1},\dots,f_{s}\rangle$ the ideal generated by
$f_{1},\dots,f_{s}$. 
Let $V$ (usually a variety of an ideal) be
a subset of $\mathbb{C}^{n}$. 
Then we set 
\[
{\bf I}\left(V\right)=\left\{ f\in\mathbb{C}\left[x_{1},\dots,x_{n}\right]\mid f\left(a_{1},\dots,a_{n}\right)=0\,\text{for all}\,\left(a_{1},\dots,a_{n}\right)\in V\right\} .
\]
We call ${\bf I}\left(V\right)$ the ideal of $V$. 
Taking the ideal
of two  affine varieties reverses containment.
So 
given two affine varieties $V,W\subset\mathbb{C}^{n}$, 
we have $V\subset W$ if and only if ${\bf I}\left(V\right)\supset{\bf I}\left(W\right)$. 

Generally,  we  include parameters 
$A_{1},\dots,A_{m}$ in the coefficients of 
 $f_{1},\dots,f_{s}$.
In this case, we say $f_{1},\dots,f_{s}\in\mathbb{C}\left(A_{1},\dots,A_{m}\right)\left[x_{1},\dots,x_{n}\right]$
are polynomials in the unknowns $x_{1},\dots,x_{n}$ with coefficients
in $\mathbb{C}\left(A_{1},\dots,A_{m}\right)$. 
By replacing
$\mathbb{C}$ with $\mathbb{C}\left(A_{1},\dots,A_{m}\right)$ we
have analogous definitions of variety and ideal over an algebraic
closure of $\mathbb{C}\left(A_{1},\dots,A_{m}\right)$.
\\

\subsection*{Dixon Matrix}
In this paper we will consider  the polynomial ring $\mathbb{C}(A_1,\dots,A_m)[x_1,x_2,x_3]$. 
In this ring, we say a polynomial $g$ is a cubic in $x_1,x_2,x_3$   if we can write 
$$g=[C]\cdot [L]$$
with $\left[L_{\,}\right]$ equal to
\[
\left[x_{1}^{3},x_{1}^{2}x_{2},x_{1}^{2}x_{3},x_{1}x_{2}^{2},x_{1}x_{2}x_{3},x_{1}x_{3}^{2},x_{2}^{3},x_{2}^{2}x_{3},x_{2}x_{3}^{2},x_{3}^{3},x_{1}^{2},x_{1}x_{2},x_{1}x_{3},x_{2}^{2},x_{2}x_{3},x_{3},x_{1},x_{2},x_{3},1\right]^{T},
\] and the entries of $[C]$  in $\mathbb{C}(A_1,\dots,A_m)$. 
The $\cdot$ is the usual dot product.  
We  say the $i$-th entry of $[C]$  is the coefficient of the $i$-th entry of  $[L]$, 
and say $[C]$  is the coefficient row vector of the cubic $g$. 

Given four quadrics
$f_1,f_2,f_3,f_4\in\mathbb{C}(A_1,\dots,A_m)[x_1,x_2,x_3]$, we show how to derive a certain list of twenty cubics, which vanish on   ${\bf V}\left(f_{1},f_{2},f_{3},f_{4}\right)$. 
From that list of twenty cubics, we construct a Dixon matrix by taking
its $i$-th row to be the coefficient row vector of the $i$-th cubic
in the list. 
Since there are twenty  monomials of degree three or less in three unknowns, the matrix we construct will be square and have a  determinant. 
The twenty cubics are chosen so that  the vanishing of this determinant  provides information about the system of four quadrics.

Specifically, the first
four cubics in our list are chosen to be $f_{1},f_{2},f_{3},f_{4}$, regarded as cubics whose first ten coefficients  are zero.
The next twelve
cubics are chosen as $x_{j}f_{i}$ where $1\leq j\leq3$, $1\leq i\leq4$.
The final four cubics in our list of twenty will be constructed below.
After defining these, we  give the definition of a $Dixon$ $matrix$ and state how it is used in the criterion.
 
To construct   the final four cubics, we first construct a matrix $U$ determined by the $f_i\in\mathbb{C}(A_1,\dots,A_m)[x_1,x_2,x_3]$. Let 
\[
U:=\left[\begin{array}{cccc}
\frac{f_{1}\left(x_{1},x_{2},x_{3}\right)-f_{1}\left(y_{1},x_{2},x_{3}\right)}{\left(x_{1}-y_{1}\right)}, & \frac{f_{1}\left(y_{1},x_{2},x_{3}\right)-f_{1}\left(y_{1},y_{2},x_{3}\right)}{\left(x_{2}-y_{2}\right)}, & \frac{f_{1}\left(y_{1},y_{2},x_{3}\right)-f_{1}\left(y_{1},y_{2},y_{3}\right)}{\left(x_{3}-y_{3}\right)}, & f_{1}\left(y_{1},y_{2},y_{3}\right)\\
\frac{f_{2}\left(x_{1},x_{2},x_{3}\right)-f_{2}\left(y_{1},x_{2},x_{3}\right)}{\left(x_{1}-y_{1}\right)}, & \frac{f_{2}\left(y_{1},x_{2},x_{3}\right)-f_{2}\left(y_{1},y_{2},x_{3}\right)}{\left(x_{2}-y_{2}\right)}, & \frac{f_{2}\left(y_{1},y_{2},x_{3}\right)-f_{2}\left(y_{1},y_{2},y_{3}\right)}{\left(x_{3}-y_{3}\right)}, & f_{2}\left(y_{1},y_{2},y_{3}\right)\\
\frac{f_{3}\left(x_{1},x_{2},x_{3}\right)-f_{3}\left(y_{1},x_{2},x_{3}\right)}{\left(x_{1}-y_{1}\right)}, & \frac{f_{3}\left(y_{1},x_{2},x_{3}\right)-f_{3}\left(y_{1},y_{2},x_{3}\right)}{\left(x_{2}-y_{2}\right)}, & \frac{f_{3}\left(y_{1},y_{2},x_{3}\right)-f_{3}\left(y_{1},y_{2},y_{3}\right)}{\left(x_{3}-y_{3}\right)}, & f_{3}\left(y_{1},y_{2},y_{3}\right)\\
\frac{f_{4}\left(x_{1},x_{2},x_{3}\right)-f_{4}\left(y_{1},x_{2},x_{3}\right)}{\left(x_{1}-y_{1}\right)}, & \frac{f_{4}\left(y_{1},x_{2},x_{3}\right)-f_{4}\left(y_{1},y_{2},x_{3}\right)}{\left(x_{2}-y_{2}\right)}, & \frac{f_{4}\left(y_{1},y_{2},x_{3}\right)-f_{4}\left(y_{1},y_{2},y_{3}\right)}{\left(x_{3}-y_{3}\right)}, & f_{4}\left(y_{1},y_{2},y_{3}\right)
\end{array}\right].
\]
Since $f_{i}\left(x_{1},x_{2},x_{3}\right)-f_{i}\left(y_{1},x_{2},x_{3}\right)$,
$f_{i}\left(y_{1},x_{2},x_{3}\right)-f_{i}\left(y_{1},y_{2},x_{3}\right)$,
and $f_{i}\left(y_{1},y_{2},x_{3}\right)-f_{i}\left(y_{1},y_{2},y_{3}\right)$ 
respectively identically vanish  when $x_{1}-y_{1}=0$, $x_{2}-y_{2}=0$, $x_{3}-y_{3}=0$,
 we can write 
\[
\begin{array}{c}
\left[\begin{array}{ccc}
f_{i}\left(x_{1},x_{2},x_{3}\right)-f_{i}\left(y_{1},x_{2},x_{3}\right), & f_{i}\left(y_{1},x_{2},x_{3}\right)-f_{i}\left(y_{1},y_{2},x_{3}\right), & f_{i}\left(y_{1},y_{2},x_{3}\right)-f_{i}\left(y_{1},y_{2},y_{3}\right)\end{array}\right]\\
=\left[\begin{array}{ccc}
\left(x_{1}-y_{1}\right) h_{i1}(x_{1},x_{2},x_{3},y_{1}), & \left(x_{2}-y_{2}\right) h_{i2}(x_{2},x_{3},y_{1},y_{2}), & \left(x_{3}-y_{3}\right) h_{i3}(x_{3},y_{1},y_{2},y_{3})\end{array}\right]
\end{array}\]
with linear forms $h_{ij}\in\mathbb{C}(A_1,\dots,A_m)[x_1,x_2,x_3,y_1,y_2,y_3]$. 
So 
\[
U=\left[\begin{array}{cccc}
h_{11}(x_{1},x_{2},x_{3},y_{1}), & h_{12}(x_{2},x_{3},y_{1},y_{2}), & h_{13}(x_{3},y_{1},y_{2},y_{3}), & f_{1}\left(y_{1},y_{2},y_{3}\right)\\
h_{21}(x_{1},x_{2},x_{3},y_{1}), & h_{22}(x_{2},x_{3},y_{1},y_{2}), & h_{23}(x_{3},y_{1},y_{2},y_{3}), & f_{2}\left(y_{1},y_{2},y_{3}\right)\\
h_{31}(x_{1},x_{2},x_{3},y_{1}), & h_{32}(x_{2},x_{3},y_{1},y_{2}), & h_{33}(x_{3},y_{1},y_{2},y_{3}), & f_{3}\left(y_{1},y_{2},y_{3}\right)\\
h_{41}(x_{1},x_{2},x_{3},y_{1}), & h_{42}(x_{2},x_{3},y_{1},y_{2}), & h_{43}(x_{3},y_{1},y_{2},y_{3}), & f_{4}\left(y_{1},y_{2},y_{3}\right)
\end{array}\right].
\]
With these two formulations of $U$ we  prove the following lemma. 
\begin{lem}
Let $f_{1},f_{2},f_{3},f_{4}\in\mathbb{C}(A_1,\dots,A_m)\left[x_{1},x_{2},x_{3}\right]$, $h_{ij}\in\mathbb{C}(A_1,\dots,A_m)\left[x_{1},x_{2},x_{3},y_1,y_2,y_3\right]$,
and $U$ be as above. 
Let $F\left(x_{1},x_{2},x_{3},y_{1},y_{2},y_{3}\right):=\det U = \sum_{i,j,k}\psi_{ijk}\cdot y_{1}^{i}y_{2}^{j}y_{3}^{k}$, with $\psi_{ijk}\in\mathbb{C}(A_1,\dots,A_m)[x_1,x_2,x_3]$.
If $f_1,f_2,f_3,f_4$  are of degree $2$ or less,
then $\psi_{ijk}$  vanish on ${\bf V}\left(f_{1},f_{2},f_{3},f_{4}\right)$ and are cubics
in the unknowns $x_{1},x_{2},x_{3}$ for $i+j+k\leq1$.
\end{lem}

\begin{proof}
Let $U_{i}$ denote the $i$-th column of $U$. 
Then, for any point $\left(s_{1},s_{2},s_{3}\right)\in{\bf V}\left(\langle f_{1},f_{2},f_{3},f_{4}\rangle\right)$,
\[
\left(s_{1}-y_{1}\right)U_{1}+\left(s_{2}-y_{2}\right)U_{2}+\left(s_{3}-y_{3}\right)U_{3}+U_{4}
\]
is identically the zero vector when $(x_1,x_2,x_3)=(s_1,s_2,s_3)$.  
To see this, we note the $i$-th entry of this vector is  
\[
\left(s_{1}-y_{1}\right) h_{i1}(s_{1},s_{2},s_{3},y_{1})
+\left(s_{2}-y_{2}\right) h_{i2}(s_{2},s_{3},y_{1},y_{2})
+\left(s_{3}-y_{3}\right) h_{i3}(s_{3},y_{1},y_{2},y_{3}) 
+ f_{i}\left(y_{1},y_{2},y_{3}\right),
\]
which equals 
\[
f_{i}\left(s_{1},s_{2},s_{3}\right)-f_{i}\left(y_{1},s_{2},s_{3}\right) + f_{i}\left(y_{1},s_{2},s_{3}\right)-f_{i}\left(y_{1},y_{2},s_{3}\right)
+f_{i}\left(y_{1},y_{2},s_{3}\right)-f_{i}\left(y_{1},y_{2},y_{3}\right)+ f_{i}\left(y_{1},y_{2},y_{3}\right) 
\]
simplifying to $f_{i}\left(s_{1},s_{2},s_{3}\right)$.  
Since $(s_1,s_2,s_3)\in {\bf V}\left(\langle f_{1},f_{2},f_{3},f_{4}\rangle\right)$ we have $f_i(s_1,s_2,s_3)=0$.
So the columns of $U$  are linearly dependent and the determinant of $U$ vanishes. 
This means that $F=\det U$ vanishes at $\left(s_{1},s_{2},s_{3}\right)$ independently
of the $y's$
and  each $\psi_{ijk}\left(x_{1},x_{2},x_{3}\right)$
must also vanish on ${\bf V}\left(f_{1},f_{2},f_{3},f_{4}\right)$.

Since each $f_{i}$ is a quadric, the first three columns of $U$ must
consist of linear entries. 
In particular, the first three columns have
entries  degree at most one in the $x's$. 
Since the last
column contains no $x's$, each entry in this column is of degree zero in the
$x's$. Thus, $F=\det U$ is at most degree three in the $x's$,
and
each $\psi_{ijk}$ is a cubic. 
\end{proof}
With the previously defined $\psi_{ijk}$ we will now be able to finish
our list of $20$ linearly independent cubics which each vanish on
${\bf V}\left(f_{1},f_{2},f_{3},f_{4}\right)$. 
Doing so allows us
to define the Dixon matrix of four quadrics. 

\begin{defn*}
Let $f_{1},f_{2},f_{3},f_{4}\in\mathbb{C}(A_1,\dots,A_m) \left[x_{1},x_{2},x_{3}\right]$ each be of degree $2$ or less. 
We define the $\emph{Dixon matrix}$ of $f_{1},f_{2},f_{3},f_{4}$
to be a matrix whose $i-th$ row is given by the coefficient vector
of $g_{i}$ where 
\[
\begin{array}{cccc}
g_{1}=f_{1}, & g_{2}=f_{2}, & g_{3}=f_{3}, & g_{4}=f_{4}\\
g_{5}=x_{1}f_{1}, & g_{6}=x_{1}f_{2}, & g_{7}=x_{1}f_{3}, & g_{8}=x_{1}f_{4}\\
g_{9}=x_{2}f_{1}, & g_{10}=x_{2}f_{2}, & g_{11}=x_{2}f_{3}, & g_{12}=x_{2}f_{4}\\
g_{13}=x_{3}f_{1}, & g_{14}=x_{3}f_{2}, & g_{15}=x_{3}f_{3}, & g_{16}=x_{3}f_{4}\\
g_{17}=\psi_{000}, & g_{18}=\psi_{101}, & g_{19}=\psi_{010}, & g_{20}=\psi_{001}
\end{array}.
\]
 We denote the Dixon matrix as $Dix\left(f_{1},f_{2},f_{3},f_{4}\right)$. 
To distinguish between parameters and unknowns we may notate the Dixon matrix as $Dix\left(f_{1},f_{2},f_{3},f_{4};x_1,x_2,x_3\right)$.

\end{defn*}
Note that permuting the order of the unknowns permutes the columns of the Dixon matrix, and permuting 
the order of the $f_{i}$'s permutes  the rows of the Dixon matrix. 
So up to
permutation of rows or columns the Dixon matrix is well defined. But
because singularity of a matrix is invariant under permutation of
rows or columns then this causes no change in the determinant (up
to scalar constant) of $Dix\left(f_{1},f_{2},f_{3},f_{4}\right).$

\begin{thm}
Suppose we are given four quadrics in three unknowns,
$$f_1,f_2,f_3\in\mathbb{C}(A_1,\dots,A_m)[x_1,x_2,x_3]$$ 
Then the Dixon matrix $Dix\left(f_{1},f_{2},f_{3},f_{4}\right)$ is singular
if ${\bf V}\left(\langle f_{1},f_{2},f_{3},f_{4}\rangle\right)\neq\emptyset$.  The determinant of the Dixon matrix vanishes if there exists a solution to the system. 
 \end{thm}
\begin{proof}
If $\left(\xi_{1},\xi_{2},\xi_{3}\right)\in{\bf V}\left(\langle f_{1},f_{2},f_{3},f_{4}\rangle\right)$
then $Dix\left(f_{1},f_{2},f_{3},f_{4}\right)$ maps 
\[
\left[\xi_{1}^{3},\xi_{1}^{2}\xi_{2},\xi_{1}^{2}\xi_{3},\xi_{1}\xi_{2}^{2},\xi_{1}\xi_{2}\xi_{3},\xi_{1}\xi_{3}^{2},\xi_{2}^{3},\xi_{2}^{2}\xi_{3},\xi_{2}\xi_{3}^{2},\xi_{3}^{3},\xi_{1}^{2},\xi_{1}\xi_{2},\xi_{1}\xi_{3},\xi_{2}^{2},\xi_{2}\xi_{3},\xi_{3},\xi_{1},\xi_{2},\xi_{3},1\right]^T
\]
 to zero. 
This is because the Dixon matrix was constructed using coefficient vectors of $g_i$, so $Dix\left(f_{1},f_{2},f_{3},f_{4}\right)\cdot [L]$ equals 
$[g_1(x_1.x_2,x_3),g_2(x_1.x_2,x_3),\cdots,g_{20}(x_1.x_2,x_3)]^T$.  
This is the zero vector when  $(x_1,x_2,x_3)=(\xi_1.\xi_2,\xi_3)\in {\bf V}\left(\langle f_{1},f_{2},f_{3},f_{4}\rangle\right)$.
So the kernel is non-trivial and the matrix is singular.
\end{proof}
The next theorem shows how to use the technique in the case where
$f_1,f_2,f_3,f_4\in\mathbb{C}[q_0,q_1,q_2,q_3]$ are quadrics in four unknowns. To define the Dixon matrix of four quadrics in four unknowns we specify three of the unknowns and consider the fourth as a parameter like $A_i$. So the determinant of 
$Dix\left(f_{1},f_{2},f_{3},f_{4};q_{1},q_2,q_3\right)$   will be an element of 
$\mathbb{C}(A_1,\dots,A_m,q_0)$. In fact, 
$Dix\left(f_{1},f_{2},f_{3},f_{4};q_{1},q_2,q_3\right)$  will be an element of 
$\mathbb{C}(A_1,\dots,A_m)[q_0]$.

\begin{thm}
Suppose we are given four quadrics in four unknowns
$$f_1,f_2,f_3,f_4\in\mathbb{C}(A_1,\dots,A_m)[q_0,q_1,q_2,q_3].$$
If $\left(s_{0},s_{1},s_{2},s_{3}\right)\in{\bf V}\left(\langle f_{1},f_{2},f_{3},f_{4}\rangle\right)$
then the univariate polynomial 
$$\det Dix\left(f_{1},f_{2},f_{3},f_{4};q_{1},q_2,q_3\right)\in\mathbb{C}(A_1,\dots,A_m)[q_0]$$ vanishes
at $s_{0}$.
\end{thm}

\begin{proof}
By setting $q_{0}$ to equal $s_{0}$, the matrix $Dix\left(f_{1},f_{2},f_{3},f_{4};q_{1},q_2,q_3\right)$  maps a nonzero vector to zero as in the previous theorem. 
This means its determinant  of $Dix\left(f_{1},f_{2},f_{3},f_{4};q_{1},q_2,q_3\right)$
must vanish at $s_0$. 
\end{proof}

With this theorem we have that ${\bf V}\left(\langle f_{1},f_{2},f_{3},f_{4}\rangle\right)\subset{\bf V}\left( \det Dix\left(f_{1},f_{2},f_{3},f_{4};q_{1},q_2,q_3\right)\right)$.
By permuting $q_0$ with the other three $q_i$ we conclude the following corollary.

\begin{cor}
If $f_1,f_2,f_3,f_4\in\mathbb{C}[q_0,q_1,q_2,q_3]$   are quadrics, then 
${\bf V}\left(\langle f_{1},f_{2},f_{3},f_{4}\rangle\right)$ is contained in 
$$\begin{array}{c}
{\bf V}(\det Dix\left(f_{1},f_{2},f_{3},f_{4};q_{1},q_{2},q_{3}\right),\det Dix\left(f_{1},f_{2},f_{3},f_{4};q_{0},q_{2},q_{3}\right),\\
\det Dix\left(f_{1},f_{2},f_{3},f_{4};q_{1},q_{0},q_{3}\right),\det Dix\left(f_{1},f_{2},f_{3},f_{4};q_{1},q_{2},q_{0}\right)).
\end{array}$$
\end{cor}

\section{Case study: The $3-\hat{\underline{R}}(RRR)_E\hat{R}$ architecture}


\begin{figure} [t]
\centering
  \includegraphics[width=90mm]{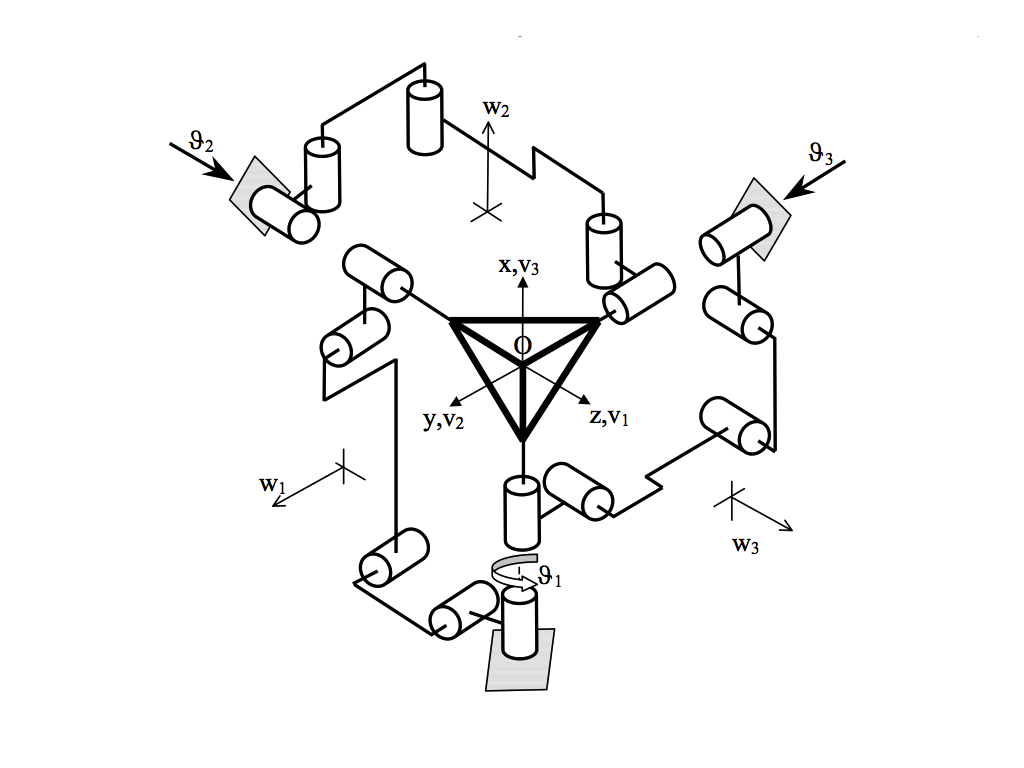}
\caption{The $3-\hat{\underline{R}}(RRR)_E\hat{R}$ SPM.}
\label{manipulator_geometry}
\end{figure}

The $3-\hat{\underline{R}}(RRR)_E\hat{R}$ SPM (Figure~\ref{manipulator_geometry}) was first proposed in~\cite{Karouia2000}. According to the reference systems in Figure~\ref{manipulator_geometry}, the kinematics problem may be rephrased as solving a system of equations: 

\[
\begin{array}{ccc}
f_{1} & = & A_{1}\left(q{}_{0}^{2}-q_{1}^{2}-q_{2}^{2}+q_{3}^{2}\right)+2B_{1}\left(q_{2}q_{3}-q_{0}q_{1}\right)\\
f_{2} & = & A_{2}\left(q{}_{0}^{2}-q_{1}^{2}+q_{2}^{2}-q_{3}^{2}\right)+2B_{2}\left(q_{1}q_{2}-q_{0}q_{3}\right)\\
f_{3} & = & A_{3}\left(q{}_{0}^{2}+q_{1}^{2}-q_{2}^{2}-q_{3}^{2}\right)+2B_{3}\left(q_{1}q_{3}-q_{0}q_{2}\right)\\
f_{4} & = & q_{0}^{2}+q_{1}^{2}+q_{2}^{2}+q_{3}^{2}-1
\end{array}
\]
with $A_i=\sin(\theta_i)$ and  $B_i=\cos(\theta_i)$ by taking 
$$\mbox{\ensuremath{\mbox{\ensuremath{\left[\begin{array}{ccc}
  &   &  \\
w_{1}  &  w_{2}  &  w_{3}\\
 &   &  
\end{array}\right]} }}}=\left[\begin{array}{ccc}
0 & B_{2} & A_{3}\\
B_{1} & A_{2} & 0\\
A_{1} & 0 & B_{3}
\end{array}\right],\mbox{ \,\ \,\  }\mbox{\ensuremath{\left[\begin{array}{ccc}
  &   &  \\
\nu_{1}  &  \nu_{2}  &  \nu_{3}\\
 &   &  
\end{array}\right]} }=\left[\begin{array}{ccc}
0 & 0 & 1\\
0 & 1 & 0\\
1 & 0 & 0
\end{array}\right]$$
in Equation (~\ref{EQ_geometrical_condition}).
The $\theta_i$'s are the motor angles.
 We say $\left(s_{0},s_{1},s_{2},s_{3}\right)\in\mathbb{C}^{4}$ is
a solution of the $3-\hat{\underline{R}}(RRR)_E\hat{R}$ system if $f_{1}\left(s_{0},s_{1},s_{2},s_{3}\right)=\cdots=f_{4}\left(s_{0},s_{1},s_{2},s_{3}\right)=0.$ 
\begin{lem}
The $q_{0}$ coordinates of the solutions to the $3-\hat{\underline{R}}(RRR)_E\hat{R}$ system are
roots of a degree $16$ polynomial in the unknown $q_{0}$
with parameters in $A_{1},A_{2},A_{3},B_{1},B_{2},B_{3}$ . This polynomial
is the determinant of a Dixon matrix $Dix\left(f_{1},f_{2},f_{3},f_{4};q_{1},q_2,q_3\right)$ and solvable in radicals:
\[
\begin{array}{l}
\det Dix\left(f_{1},f_{2},f_{3},f_{4};q_{1},q_2,q_3\right)  =  \left(2q_{0}-1\right)^{4}\left(2q_{0}+1\right)^{4}\cdot G\left(q_{0}\right)\\
G\left(q_0\right)  =  -16777216w^{2}q_0^{8}+16777216w^{2}q_0^{6}-2097152w\cdot w_{1}q_0^{4}+w\cdot w_{2}q_0^{2}+w_{3}
\end{array}
\]
with {\tiny 
\[
\begin{array}{lll}
w & = & (A_{2}^{2}A_{3}^{2}+A_{2}^{2}B_{3}^{2}+B_{2}^{2}B_{3}^{2})(A_{1}^{2}A_{3}^{2}+B_{1}^{2}A_{3}^{2}+B_{1}^{2}B_{3}^{2})(A_{1}^{2}A_{2}^{2}+A_{1}^{2}B_{2}^{2}+B_{1}^{2}B_{2}^{2})\\
w_{1} & = & (3A_{1}^{4}A_{2}^{4}A_{3}^{4}+2A_{1}^{2}B_{1}^{2}A_{2}^{4}A_{3}^{4}+2A_{1}^{4}A_{2}^{2}B_{2}^{2}A_{3}^{4}+5A_{1}^{2}B_{1}^{2}A_{2}^{2}B_{2}^{2}A_{3}^{4}+3B_{1}^{4}A_{2}^{2}B_{2}^{2}A_{3}^{4}+2A_{1}B_{1}^{3}A_{2}^{3}B_{2}A_{3}^{3}B_{3}+2A_{1}^{3}B_{1}A_{2}B_{2}^{3}A_{3}^{3}B_{3}+\\
 &  & +3A_{1}^{4}B_{2}^{4}A_{3}^{2}B_{3}^{2}+2A_{1}B_{1}^{3}A_{2}B_{2}^{3}A_{3}^{3}B_{3}+2A_{1}^{4}A_{2}^{4}A_{3}^{2}B_{3}^{2}+5A_{1}^{2}B_{1}^{2}A_{2}^{4}A_{3}^{2}B_{3}^{2}+5A_{1}^{4}A_{2}^{2}B_{2}^{2}A_{3}^{2}B_{3}^{2}+9A_{1}^{2}B_{1}^{2}A_{2}^{2}B_{2}^{2}A_{3}^{2}B_{3}^{2}+\\
 &  & +5B_{1}^{4}A_{2}^{2}B_{2}^{2}A_{3}^{2}B_{3}^{2}+5A_{1}^{2}B_{1}^{2}B_{2}^{4}A_{3}^{2}B_{3}^{2}+2B_{1}^{4}B_{2}^{4}A_{3}^{2}B_{3}^{2}+2A_{1}^{3}B_{1}A_{2}^{3}B_{2}A_{3}B_{3}^{3}+2A_{1}B_{1}^{3}A_{2}^{3}B_{2}A_{3}B_{3}^{3}+2A_{1}^{3}B_{1}A_{2}B_{2}^{3}A_{3}B_{3}^{3}+\\
 &  & +6A_{1}B_{1}^{3}A_{2}B_{2}^{3}A_{3}B_{3}^{3}+3A_{1}^{2}B_{1}^{2}A_{2}^{4}B_{3}^{4}+5A_{1}^{2}B_{1}^{2}A_{2}^{2}B_{2}^{2}B_{3}^{4}+2B_{1}^{4}A_{2}^{2}B_{2}^{2}B_{3}^{4}+2A_{1}^{2}B_{1}^{2}B_{2}^{4}B_{3}^{4})\\
w_{2} & = & (A_{1}^{4}A_{2}^{4}A_{3}^{4}+A_{1}^{2}B_{1}^{2}A_{2}^{2}B_{2}^{2}A_{3}^{4}+B_{1}^{4}A_{2}^{2}B_{2}^{2}A_{3}^{4}-2A_{1}^{3}B_{1}A_{2}^{3}B_{2}A_{3}^{3}B_{3}+2A_{1}B_{1}^{3}A_{2}^{3}B_{2}A_{3}^{3}B_{3}+2A_{1}^{3}B_{1}A_{2}B_{2}^{3}A_{3}^{3}B_{3}+\\
 &  & +2A_{1}B_{1}^{3}A_{2}B_{2}^{3}A_{3}^{3}B_{3}+A_{1}^{2}B_{1}^{2}A_{2}^{4}A_{3}^{2}B_{3}^{2}+A_{1}^{4}A_{2}^{2}B_{2}^{2}A_{3}^{2}B_{3}^{2}+7A_{1}^{2}B_{1}^{2}A_{2}^{2}B_{2}^{2}A_{3}^{2}B_{3}^{2}+B_{1}^{4}A_{2}^{2}B_{2}^{2}A_{3}^{2}B_{3}^{2}+A_{1}^{4}B_{2}^{4}A_{3}^{2}B_{3}^{2}+\\
 &  & +A_{1}^{2}B_{1}^{2}B_{2}^{4}A_{3}^{2}B_{3}^{2}+2A_{1}^{3}B_{1}A_{2}^{3}B_{2}A_{3}B_{3}^{3}2A_{1}B_{1}^{3}A_{2}^{3}B_{2}A_{3}B_{3}^{3}+2A_{1}^{3}B_{1}A_{2}B_{2}^{3}A_{3}B_{3}^{3}+A_{1}^{2}B_{1}^{2}A_{2}^{4}B_{3}^{4}+A_{1}^{2}B_{1}^{2}A_{2}^{2}B_{2}^{2}B_{3}^{4})\\
w_{3} & = & (-A_{1}^{4}A_{2}^{4}A_{3}^{4}+A_{1}^{2}B_{1}^{2}A_{2}^{2}B_{2}^{2}A_{3}^{4}+B_{1}^{4}A_{2}^{2}B_{2}^{2}A_{3}^{4}+4A_{1}^{3}B_{1}A_{2}^{3}B_{2}A_{3}^{3}B_{3}+2A_{1}B_{1}^{3}A_{2}^{3}B_{2}A_{3}^{3}B_{3}+\\
 &  & +2A_{1}^{3}B_{1}A_{2}B_{2}^{3}A_{3}^{3}B_{3}+A_{1}^{4}B_{2}^{4}A_{3}^{2}B_{3}^{2}+2A_{1}B_{1}^{3}A_{2}B_{2}^{3}A_{3}^{3}B_{3}+A_{1}^{2}B_{1}^{2}A_{2}^{4}A_{3}^{2}B_{3}^{2}+A_{1}^{4}A_{2}^{2}B_{2}^{2}A_{3}^{2}B_{3}^{2}+A_{1}^{2}B_{1}^{2}A_{2}^{2}B_{2}^{2}A_{3}^{2}B_{3}^{2}+\\
 &  & +B_{1}^{4}A_{2}^{2}B_{2}^{2}A_{3}^{2}B_{3}^{2}+A_{1}^{2}B_{1}^{2}B_{2}^{4}A_{3}^{2}B_{3}^{2}+2A_{1}^{3}B_{1}A_{2}^{3}B_{2}A_{3}B_{3}^{3}+2A_{1}B_{1}^{3}A_{2}^{3}B_{2}A_{3}B_{3}^{3}+2A_{1}^{3}B_{1}A_{2}B_{2}^{3}A_{3}B_{3}^{3}+2A_{1}B_{1}^{3}A_{2}B_{2}^{3}A_{3}B_{3}^{3}+\\
 &  & +A_{1}^{2}B_{1}^{2}A_{2}^{4}B_{3}^{4}+A_{1}^{2}B_{1}^{2}A_{2}^{2}B_{2}^{2}B_{3}^{4})
\end{array}
\]
} \end{lem}
\begin{proof}
We want to construct the Dixon matrix $Dix\left(f_{1},f_{2},f_{3},f_{4};q_{1},q_{2},q_{3}\right)$
of the system with respect to $q_{0}$ as its determinant vanishes on the $q_0$   coordinates of the solutions.
We easily get the entries
of the first $16$ rows from the coefficients of $f_{1},q_{1}f_{2},q_{2}f_{1},q_{3}f_{1},\dots,f_{4},q_{1}f_{4},q_{2}f_{4},q_{3}f_{4}$,
and the last $4$ rows come from the coefficients of $\psi_{000},\psi_{001},\psi_{010},\psi_{100}$.
To compute these we take the determinant of the matrix
\begin{equation}
U:=\left[\begin{array}{cccc}
\frac{f_{1}\left(x_{1},x_{2},x_{3}\right)-f_{1}\left(y_{1},x_{2},x_{3}\right)}{\left(x_{1}-y_{1}\right)}, & \frac{f_{1}\left(y_{1},x_{2},x_{3}\right)-f_{1}\left(y_{1},y_{2},x_{3}\right)}{\left(x_{2}-y_{2}\right)}, & \frac{f_{1}\left(y_{1},y_{2},x_{3}\right)-f_{1}\left(y_{1},y_{2},y_{3}\right)}{\left(x_{3}-y_{3}\right)}, & f_{1}\left(y_{1},y_{2},y_{3}\right)\\
\frac{f_{2}\left(x_{1},x_{2},x_{3}\right)-f_{2}\left(y_{1},x_{2},x_{3}\right)}{\left(x_{1}-y_{1}\right)}, & \frac{f_{2}\left(y_{1},x_{2},x_{3}\right)-f_{2}\left(y_{1},y_{2},x_{3}\right)}{\left(x_{2}-y_{2}\right)}, & \frac{f_{2}\left(y_{1},y_{2},x_{3}\right)-f_{2}\left(y_{1},y_{2},y_{3}\right)}{\left(x_{3}-y_{3}\right)}, & f_{2}\left(y_{1},y_{2},y_{3}\right)\\
\frac{f_{3}\left(x_{1},x_{2},x_{3}\right)-f_{3}\left(y_{1},x_{2},x_{3}\right)}{\left(x_{1}-y_{1}\right)}, & \frac{f_{3}\left(y_{1},x_{2},x_{3}\right)-f_{3}\left(y_{1},y_{2},x_{3}\right)}{\left(x_{2}-y_{2}\right)}, & \frac{f_{3}\left(y_{1},y_{2},x_{3}\right)-f_{3}\left(y_{1},y_{2},y_{3}\right)}{\left(x_{3}-y_{3}\right)}, & f_{3}\left(y_{1},y_{2},y_{3}\right)\\
\frac{f_{4}\left(x_{1},x_{2},x_{3}\right)-f_{4}\left(y_{1},x_{2},x_{3}\right)}{\left(x_{1}-y_{1}\right)}, & \frac{f_{4}\left(y_{1},x_{2},x_{3}\right)-f_{4}\left(y_{1},y_{2},x_{3}\right)}{\left(x_{2}-y_{2}\right)}, & \frac{f_{4}\left(y_{1},y_{2},x_{3}\right)-f_{4}\left(y_{1},y_{2},y_{3}\right)}{\left(x_{3}-y_{3}\right)}, & f_{4}\left(y_{1},y_{2},y_{3}\right)
\end{array}\right]. \nonumber
\end{equation}
We write $\det U=\sum\psi_{ijk}(q_{1},q_{2},q_{3},A_{1},A_{2},A_{3},B_{1},B_{2},B_{3},q_{0})a^{i}b^{j}c^{k}$.
Since  we are only interested in the $\psi$ where $i+j+k\leq1$ then to increase computation performance we take $a^{2}=b^{2}=c^{2}=ab=ac=bc=0$. The exact formula was found using Macaulay2.
We now
have a matrix whose determinant is up to scalar, $\det Dix\left(f_{1},f_{2},f_{3},f_{4};q_{1},q_2,q_3\right)=\left(2q_{0}-1\right)^{4}\left(2q_{0}+1\right)^{4}\cdot G\left(q_{0}\right)$.
\end{proof}

Repeating the computations for $1\leq i\leq3$, we see there are symmetries in our system. As up to scalar and change of variables,
$ \det Dix\left(f_{1},f_{2},f_{3},f_{4};q_{1},q_2,q_3\right)$,
$ \det Dix\left(f_{1},f_{2},f_{3},f_{4};q_{0},q_2,q_3\right)$,
$ \det Dix\left(f_{1},f_{2},f_{3},f_{4};q_{1},q_0,q_3\right)$, 
and
$ \det Dix\left(f_{1},f_{2},f_{3},f_{4};q_{1},q_2,q_0\right)$ 
are equal. 
These calculations give us insight to prove the next lemma.

\begin{lem}
If $\sigma_{1}=\left(s_{0},s_{1},s_{2},s_{3}\right)\in{\bf V}\left(\langle f_{1},f_{2},f_{3},f_{4}\rangle\right)$
then 
\[
{\footnotesize
\begin{array}{cccc}
\sigma_{1}=\left(s_{0},s_{1},s_{2},s_{3}\right) & \sigma_{2}=\left(s_{1},-s_{0},-s_{3},s_{2}\right) & \sigma_{3}=\left(s_{2},s_{3},-s_{0},-s_{1}\right) & \sigma_{4}=\left(s_{3},-s_{2},s_{1},-s_{0}\right)\\
-\sigma_{1}=\left(-s_{0},-s_{1},-s_{2},-s_{3}\right) & -\sigma_{2}=\left(-s_{1},s_{0},s_{3},-s_{2}\right) & -\sigma_{3}=\left(-s_{2},-s_{3},s_{0},s_{1}\right) & -\sigma_{4}=\left(-s_{3},s_{2},-s_{1},s_{0}\right)
\end{array}
}
\]
are solutions to the $3-\hat{\underline{R}}(RRR)_E\hat{R}$ system. 
In addition The $3-\hat{\underline{R}}(RRR)_E\hat{R}$ system
has $8$ distinct extraneous solutions. 
\end{lem}
\begin{proof}
We show that if $\sigma_{1}=\left(s_{0},s_{1},s_{2},s_{3}\right)$
is a solution, then $\sigma_{2}=\left(s_{1},-s_{0},-s_{3},s_{2}\right)$
is a solution. Substituting $\sigma_{2}$ into our system gives 
\[
\begin{array}{ccc}
f_{1}\left(\sigma_{2}\right) & = & A_{1}\left((s_{1}){}^{2}-(-s_{0})^{2}-(-s_{3})^{2}+(s_{2})^{2}\right)+B_{1}2\left((-s_{3})s_{2}-s_{1}(-s_{0})\right)\\
f_{2}\left(\sigma_{2}\right) & = & A_{2}\left((s_{1}){}^{2}-(-s_{0})^{2}+(-s_{3})^{2}-(s_{2})^{2}\right)+B_{2}2\left((-s_{0})(-s_{3})-s_{1}s_{2}\right)\\
f_{3}\left(\sigma_{2}\right) & = & A_{3}\left((s_{1}){}^{2}+(-s_{0})^{2}-(-s_{3})^{2}-(s_{2})^{2}\right)+B_{3}2\left((-s_{0})s_{2}-s_{1}(-s_{3})\right)\\
f_{4}\left(\sigma_{2}\right) & = & (s_{1}){}^{2}+(-s_{0})^{2}+(-s_{3})^{2}+(s_{2})^{2}-1
\end{array}
\]
But after simplifying these equations we see $f_{1}\left(\sigma_{1}\right)=-f_{1}\left(\sigma_{2}\right)$,
$f_{2}\left(\sigma_{1}\right)=-f_{2}\left(\sigma_{2}\right)$ $f_{3}\left(\sigma_{1}\right)=f_{3}\left(\sigma_{2}\right)$,
and $f_{4}\left(\sigma_{1}\right)=f_{4}\left(\sigma_{2}\right)$.
So $\sigma_{2}$ is a solution. A similar argument shows the remaining
cases.

Now, we show the $3-\hat{\underline{R}}(RRR)_E\hat{R}$ system has $8$ distinct extraneous solutions.
Invariant of our choice of parameters the $3-\hat{\underline{R}}(RRR)_E\hat{R}$ system has an extraneous
solution $\rho_{1}=\left(\frac{1}{2},\frac{1}{2},\frac{1}{2},\frac{1}{2}\right)$.
But because it has one, it must have eight:
\[
{\footnotesize
\begin{array}{cccc}
\rho_{1}=\left(\frac{1}{2},\frac{1}{2},\frac{1}{2},\frac{1}{2}\right) & \rho_{2}=\left(\frac{1}{2},-\frac{1}{2},-\frac{1}{2},\frac{1}{2}\right) & \rho_{3}=\left(\frac{1}{2},\frac{1}{2},-\frac{1}{2},-\frac{1}{2}\right) & \rho_{4}=\left(\frac{1}{2},-\frac{1}{2},\frac{1}{2},-\frac{1}{2}\right)\\
-\rho_{1}=\left(-\frac{1}{2},-\frac{1}{2},-\frac{1}{2},-\frac{1}{2}\right) & -\rho_{2}=\left(-\frac{1}{2},\frac{1}{2},\frac{1}{2},-\frac{1}{2}\right) & -\rho_{3}=\left(-\frac{1}{2},-\frac{1}{2},\frac{1}{2},\frac{1}{2}\right) & -\rho_{4}=\left(-\frac{1}{2},\frac{1}{2},-\frac{1}{2},\frac{1}{2}\right)
\end{array}
}
\]

\end{proof}
With these results we can solve the $3-\hat{\underline{R}}(RRR)_E\hat{R}$ system in radicals: first, determine
the coordinates of the solutions using Dixon's determinant; second,
through substitutions find a non-extraneous solution.

\subsection{Numerical Example}

Specifying the coefficients in the case study as
\[
\begin{array}{cc}
A_{1}=3/5, & B_{1}=4/5\\
A_{2}=5/13, & B_{2}=12/13\\
A_{3}=7/25, & B_{3}=24/25
\end{array}
\]
then this instance of the $3-\hat{\underline{R}}(RRR)_E\hat{R}$ system has non extraneous solutions
\[	
\begin{array}{cccc}
\left(r_{1}r_{2},-r_{3},-r_{4}\right) & \left(r_{2},-r_{1},r_{4},-r_{3}\right) & \left(-r_{3},-r_{4},-r_{1},-r_{2}\right) & \left(-r_{4},r_{3},r_{2},-r_{1}\right)\\
\left(-r_{1},-r_{2},r_{3},r_{4}\right) & \left(-r_{2},r_{1},-r_{4},r_{3}\right) & \left(r_{3},r_{4},r_{1},r_{2}\right) & \left(r_{4},-r_{3},-r_{2},r_{1}\right)
\end{array}
\]
 with 
\[
\begin{array}{ccccc}
r_{1} & = & \frac{1}{\sqrt{20}}\sqrt{5-\frac{125640}{\sqrt{1029158929}}-1396\sqrt{\frac{2\left(96269-3\sqrt{1029158929}\right)}{17495701793}}} & = & 0.22420547189459831634\\
r_{2} & = & \frac{1}{\sqrt{20}}\sqrt{5-\frac{125640}{\sqrt{1029158929}}+1396\sqrt{\frac{2\left(96269-3\sqrt{1029158929}\right)}{17495701793}}} & = & 0.24102325734564694455\\
r_{3} & = & \sqrt{\frac{1}{4}+\frac{6282}{\sqrt{1029159020}}+\frac{1}{2}\sqrt{\frac{93805283752}{437392544825}+\frac{2923224}{425\sqrt{1029158929}}}} & = & 0.87935205040860456901\\
r_{4} & = & \sqrt{\frac{1}{4}+\frac{6282}{\sqrt{1029159020}}-\frac{1}{2}\sqrt{\frac{93805283752}{437392544825}+\frac{2923224}{425\sqrt{1029158929}}}} & = & 0.34406346396151611679
\end{array}
\]

\begin{proof}
Taking the determinant of the Dixon matrix allows us to solve for
the coordinates of the solution using Cardano's formulas. The roots are 
\[
\left\{ \pm r_{1},\pm r_{2},\pm r_{3},\pm r_{4},\pm\frac{1}{2}\right\}. 
\]
Each coordinate of a  non-extraneous solution will be an element of this set. 
We disregard the possibilities when  $\pm\frac{1}{2}$ is a coordinates as these correspond to extraneous solutions. 
With our method, we have found a finite list of possible non extraneous solutions.
With this finite list and after making substitutions, we will find 
\[
\left(q_{0},q_{1},q_{2},q_{3}\right)=\left(r_{1},r_{2},-r_{3},-r_{4}\right)
\]
is a non extraneous solution. 
Because we have exact solutions, when we make the  substitution  in each of the $f_i$ we have 
\[
f_1\left(r_{1},r_{2},-r_{3},-r_{4}\right)=f_2\left(r_{1},r_{2},-r_{3},-r_{4}\right)=
f_3\left(r_{1},r_{2},-r_{3},-r_{4}\right)
=f_4\left(r_{1},r_{2},-r_{3},-r_{4}\right)=0.
\]
By finding one non-extraneous solution we can determine all eight non extraneous solutions 
using Lemma 6, solving our system.
\end{proof}
\section{Conclusions}
The forward-displacement problem of spherical parallel robots that obey the geometric condition $\mathbf{w}_i^T\mathbf{v}_i=\mathcal{C}_i~~~(i=1,2,3) $ was revisited in this paper. The method proposed uses the quaternion algebra to express the solving system of equations and the Dixon determinant procedure to solve it. 
The method was applied to the $3-\hat{\underline{R}}(RRR)_E\hat{R}$ architecture. According to previous works using different procedures, we found that the forward-displacement problem has sixteen solutions with eight  distinct non-extraneous solutions.  

\end{document}